\documentclass[reqno]{amsart}

\usepackage{hyperref}
\usepackage{amsmath,amsfonts,amssymb,amsthm}
\usepackage{abbrevs}
\usepackage{enumitem}
\usepackage{calc}
\usepackage{aliascnt}
\usepackage{mathtools}
\usepackage{lmodern}
\usepackage{xfrac}
\usepackage{cleveref}
\usepackage{blindtext}
\usepackage{macros}
\usepackage{mathrsfs}
\usepackage{calc}
\usepackage{bm}
\usepackage{accents}
\usepackage[utf8]{inputenc}
\usepackage[english]{babel}
\usepackage{indentfirst}
\usepackage{comment}
\usepackage[square,sort,comma,numbers]{natbib}
\usepackage[toc,page]{appendix}

\usepackage{setspace}

\begin{document}
\everymath{\displaystyle}

 \onehalfspacing

\title[]{
Nondegeneracy of positive bubble solutions for generalized energy-critical Hartree equations}

\author[]{Xuemei Li}
\address{\hskip-1.15em Xuemei Li
	\hfill\newline Laboratory of Mathematics and Complex Systems,
\hfill\newline Ministry of Education,
\hfill\newline School of Mathematical Sciences,
\hfill\newline Beijing Normal University,
\hfill\newline Beijing, 100875, People's Republic of China.}
\email{xuemei\_li@mail.bnu.edu.cn}

\author[]{Chenxi Liu}
\address{\hskip-1.15em Chenxi Liu
	\hfill\newline Laboratory of Mathematics and Complex Systems,
\hfill\newline Ministry of Education,
\hfill\newline School of Mathematical Sciences,
\hfill\newline Beijing Normal University,
\hfill\newline Beijing, 100875, People's Republic of China.}
\email{liuchenxi@mail.bnu.edu.cn}

\author[]{Xingdong Tang}
\address{\hskip-1.15em Xingdong Tang
	\hfill\newline School of Mathematics and Statistics, \hfill\newline Nanjing Univeristy of Information Science and Technology,
	\hfill\newline Nanjing, 210044,  People's Republic of China.}
\email{txd@nuist.edu.cn}

\author[]{Guixiang Xu}
\address{\hskip-1.15em Guixiang Xu
	\hfill\newline Laboratory of Mathematics and Complex Systems,
	\hfill\newline Ministry of Education,
	\hfill\newline School of Mathematical Sciences,
	\hfill\newline Beijing Normal University,
	\hfill\newline Beijing, 100875, People's Republic of China.}
\email{guixiang@bnu.edu.cn}



\begin{abstract}
In this paper,   we show the nondegeneracy of positive bubble solutions  for generalized energy-critical
Hartree equations (NLH)
\begin{equation*}
  -{\Delta u}\sts{x}
  -{\bm\alpha}\sts{N,\lambda}
  \int_{\R^N}
  { \frac{ u^{p}\sts{y}}{\pabs{\,x-y\,}{\lambda}} }\diff{y}\,
  u^{p-1}\sts{x}
  =0,\quad x\in \R^N
\end{equation*}
where  $N\geq 3$, $0<\lambda<N$, $p=\frac{2N-\lambda}{N-2}$ and ${\bm\alpha}\sts{N,\lambda}$ is a normalized constant such that $ u(x)=\left(1+|x|^2\right)^{-\frac{N-2}{2} }$ is a bubble solution of the equation \eqref{NLH}.  It solves an open nondegeneracy problem in \cite{MWX:Hartree,  GMYZ2022cvpde} and generalizes the partial nondegeneracy results  in \cite{DY2019dcds, GWY2020na, LTX2021} to the full range $0<\lambda<N$.  The key observation is that  by use of  the stereographic projection $\mathcal{S}$,  the weighted pushforward map $\mathcal{S}_*$  is  one-to-one map between the null space of the linearized operator and the spherical harmonic function subspace $\mathcal{H}_1^{N+1}$ of degree one.

\end{abstract}

\maketitle


\section{Introduction}
\label{sec:introduction}
In this paper we consider the generalized energy-critical Hartree equations in $\R^N$
\begin{equation}
\tag{NLH}
\label{NLH}
  -{\Delta u}\sts{x}
  -{\bm\alpha}\sts{N,\lambda}
  \int_{\R^N}
  { \frac{ u^{p}\sts{y}}{\pabs{\,x-y\,}{\lambda}} }\diff{y}\,
  u^{p-1}\sts{x}
  =0,\quad x\in \R^N
\end{equation}
where $u$ is a real-valued function, $N\geq 3$, $0<\lambda<N$, $p=\frac{2N-\lambda}{N-2}$ and
the normalized constant \begin{equation}
\label{alpha}
  {\bm\alpha}\sts{N,\lambda}
  = \frac{ N\sts{N-2}}{\pi^{\frac{N}{2}}} \cdot 
\frac{
   \fct{\Gamma}{N-\frac{\lambda}{2}}
}{
    \fct{\Gamma}{\frac{N-\lambda}{2}}
}.
\end{equation}
The equation \eqref{NLH} is left invariant under the scaling transform
\begin{equation*}
u(x) \longmapsto u_\delta(x)=\delta^{\frac{N-2}{2}} u(\delta x),
\end{equation*}
which preserves the $\dot H^1(\R^N)$ norm. That is the reason why the equation \eqref{NLH} is called the energy-critical equation.

The equation \eqref{NLH},
which is also called nonlinear Choquard or Choquard-Pekar equation,
has several physical motivations.
In the subcritical case $N=3,$ $p=2$ and $\lambda=1$, the equation \eqref{NLH}
firstly appeared
in the context of Fr\"ohlich and Pekar's polaron model,
which describes the interaction between one single electron and the dielectric polarisable continuum, see \cite{FF1937PRSL,F1954AP,AM1996PB,Pekar1963}.
Later, Choquard proposed the equation \eqref{NLH}  to describe
an approximation to the Hartree-Fock theory of a plasma, and then attracted the substantial attention in the field of nonlinear elliptic equations, see \cite{Lieb1977,Lions1980na,VJ2017jfpta}. The equation \eqref{NLH} also arises as a model problem in the study of stationary solutions to nonlinear Schr\"odinger equation with nonlocal nonlinearity:
\begin{equation}\label{gNLH}
  i u_t-{\Delta u}
  -{\bm\alpha}\sts{N,\lambda}
  \int_{\R^N}
  { \frac{\abs{u\sts{y}}^{p}}{\pabs{\,x-y\,}{\lambda}} }\diff{y}\,
 \abs{u}^{p-2} u
  =0, \quad p\geq 2.
\end{equation}
Physically, the equation
\eqref{gNLH} effectively describes the mean field limit of
quantum many-body systems, see e.g., \cite{FYY2002cmp,LNR2014am,FL2003}, and references therein.

The existence and uniqueness of positive bubble solutions of the equation \eqref{NLH} has been known for some time, see e.g., \cite{DY2019dcds,GY2018,GHPS2019cvpde,MWX:Hartree,MXZ09: rad,VJ2017jfpta} and references therein.

The existence of bubble solutions of the equation \eqref{NLH}
is closely related to the sharp \HLS inequality in $\R^N$
\begin{align}\label{HLS:quote}
\sts{
\rint\rint
\frac{
    \pabs{\fct{f}{x}}{p}\pabs{\fct{f}{y}}{p}
}{
    \pabs{x-y}{\lambda}
}\diff{x}\diff{y}
}^{\frac{1}{p}}
\leq\;  \Const{C}\sts{N,\lambda}
  \norm{\nabla f}{L^2(\R^N)}^2,
\end{align}
where $N\geq 3$, $0<\lambda<N$ and $p=\frac{2N-\lambda}{N-2}$. In fact, by use of the  sharp \HLS inequality and the sharp \Sb inequality in $\R^N$ in \cite{lieb2001analysis}, Miao, Wu and the fourth author firstly showed  that the sharp constant
$\Const{C}\sts{N,\lambda}$ in \eqref{HLS:quote} is obtained in the classical energy-critical case $p=2$, $\lambda=4<N$ if and only if
\begin{equation}\label{f}
\fct{f}{x}=c\cdot \sts{\delta^2+ \pabs{x-x_0}{2}}^{-\frac{N-2}{2}}
\end{equation}
for some $c\in\R\setminus\{0\},~\delta>0$ and $x_0\in\R^N$ in \cite{MWX:Hartree}. Later, Du, Yang and Gao generalized the result in the general energy-critical case  in \cite{DY2019dcds,GY2018}. 

The existence of the extremizer for the sharp Hardy-Littlewood-Sobolev inequality \eqref{HLS:quote} is more subtle than the fact that the inequality \eqref{HLS:quote} holds.  The rearrangement inequalities, the conformal transform and the stereographic projection are useful arguements to show the existence of the extremizer of \eqref{HLS:quote}, see \cite{Lieb1983am, lieb2001analysis}.  In fact, we have

\begin{theo}{\cite{FL2012, GY2018, MWX:Hartree}}
\label{theo1}
Let $N\geq 3$, 
$0<\lambda<N$,
and $p=\frac{2N-\lambda}{N-2}$.
Then for any $f, g \in \dot{H}^1\sts{\R^N}\setminus\ltl{0}$, the inequality
\begin{equation}\label{critical HLS inequality}
\sts{
\rint\rint
\frac{
    \pabs{\fct{f}{x}}{p}\pabs{\fct{g}{y}}{p}
}{
    \pabs{x-y}{\lambda}
}\diff{x}\diff{y}
}^{\frac{1}{p}}
\leq
\Const{C}\sts{N,\lambda}
\norm{\nabla f}{2}\norm{\nabla g}{2}
\end{equation}
holds with sharp constant
\begin{equation}\label{C}
  \Const{C}\sts{N,\lambda}
=
\sts{
\frac{
        \fct{\Gamma}{N}
    }{
        \fct{\Gamma}{\frac{N}{2}}
        \sts{4\pi}^{\frac{N}{2}}
    }
}
\sts{
\frac{
    \fct{\Gamma}{\frac{N}{2}}\fct{\Gamma}{\frac{N-\lambda}{2}}
}{
    \fct{\Gamma}{{N}}\fct{\Gamma}{N-\frac{\lambda}{2}}
}
\sts{4\pi}^{N}
}^{\frac{N-2}{2N-\lambda}}.
\end{equation}
Moreover, the equality in \eqref{critical HLS inequality} holds if and only if
\begin{equation*}
  f(x)=c\cdot \sts{\delta^2+ \pabs{x-x_0}{2}}^{-\frac{N-2}{2}},
  \qtq{and}
  \fct{g}{x}=c'\cdot f(x),
\end{equation*}
where $c,c'\in \R\setminus\ltl{0}$, $\delta>0$, and $x_0\in\R^N$.
\end{theo}

%
%

As for the rigidity classification of the positive solution to the equation \eqref{NLH} in $L^{\frac{2N}{N-2}}(\R^N)$,  Miao, Wu and the fourth author firstly showed that any nontrivial solutions to  the equation \eqref{NLH}  with constant sign in the case $p=2$, $\lambda=4<N$ must be the form of \eqref{f} by use of the Kelvin transform and the moving plane method in \cite{MWX:Hartree}. Later, Du, Yang \cite{DY2019dcds} and Guo, Hu, Peng, Shuai  \cite{GHPS2019cvpde} independently generalized the result in the general case. More precisely, we have
\begin{prop}[\cite{DY2019dcds,GHPS2019cvpde,MWX:Hartree}]
	Suppose that $N\geq 3$,  $0<\lambda<N$ and $p=\frac{2N-\lambda}{N-2}$. Let $u$ be a nontrivial solution of the equation \eqref{NLH} with constant sign, then there exist $c\in\R\setminus\{0\},~\delta>0$ and $x_0\in\R^N$
	such that $\fct{u}{x}=\fct{U_{c,\delta,x_0}}{x}$, where
	\begin{equation}\label{u:all}
	\fct{U_{c,\delta,x_0}}{x} = c\cdot \sts{\delta^2+\pabs{x-x_0}{2}}^{-\frac{N-2}{2}}.
	\end{equation}
\end{prop}

Motivated by the nondegeneracy results of eigenfunctions and ground state in \cite{FrankL:ActaMath, FJL2008, tao:book, Teschl:book},  a natural question arises
in the study of bubble solutions to the equation \eqref{NLH} is that
\begin{center}
	\textit{Are the positive solutions of
		\eqref{NLH}   non-degenerate? }
\end{center}

Since  the equation \eqref{NLH} is invariant
under the scaling and spatial translations, i. e., 
any solution ${v}$ solves the equation \eqref{NLH} if and only if $ \fct{v_{\delta,x_0}}{x}
=\delta^{\frac{N-2}{2}}
\fct{v}{\delta x+x_0}$ satisfies
\begin{equation}
\tag{NLH*}
\label{NLH:scale}
  -{\Laplace v_{\delta,x_0}}\sts{x}
  -{\bm\alpha}\sts{N,\lambda}
  \int_{\R^N}
  { \frac{ v_{\delta,x_0}^{p}\sts{y}}{\pabs{\,x-y\,}{\lambda}} }\diff{y}\,
  v_{\delta,x_0}^{p-1}\sts{x}
  =0
\end{equation}
for any $\delta>0$ and $x_0\in \R^N$.
Hence, for simplicity, it suffices to consider the solution \eqref{u:all} with
the normalized parameters $c=1$,  $\delta=1$ and $x_0=0,$ i.e.
\begin{equation}\label{u}
  \fct{u}{x}
  :=\fct{U_{1,1,0}}{x}
  =\sts{1+\abs{x}^2}^{-\frac{N-2}{2}}.
\end{equation}
That is the reason why we choose the normalized constant ${\bm\alpha}\sts{N,\lambda}$ in \eqref{NLH} and \eqref{NLH:scale}.  By differentiating \eqref{NLH:scale} with respect to
$\delta$ and $x_0$ at $\sts{\delta,x_0}=\sts{1,0}$, we know that
the generator
\begin{align}
\label{null 1}
\fct{\varphi_{j}}{x}:=\fct{\frac{\partial u}{\partial x_j}}{x}
=
\sts{2-N}\fct{u}{x}\frac{x_j}{1+\pabs{x}{2}},\;\;
~
1\leq j\leq N,
\end{align}
and
\begin{align}
\label{null 2}
\fct{\varphi_{N+1}}{x}
  :=
\frac{N-2}{2}\fct{u}{x}+x\cdot\fct{\nabla u}{x}
=\frac{N-2}{2}\fct{u}{x}\frac{1-\pabs{x}{2}}{1+\pabs{x}{2}}
,
\end{align}
are $N+1$ linear independent bounded  solutions  with vanishing at infinity
to the following linearized equation
\begin{align}
\label{linear equation}
  -{\Laplace \varphi }
  ={\bm\alpha}\sts{N,\lambda}
    \left[{p\sts{ \frac{1}{\pabs{\,\cdot\,}{\lambda}}
    \ast u^{p-1}\varphi }
    u^{p-1}
  +
  \sts{p-1}\sts{ \frac{1}{\pabs{\,\cdot\,}{\lambda}}\ast u^{p} }
  u^{p-2}\varphi
}\right].
\end{align}

\begin{defi}
	The solution $u$ defined by \eqref{u} to the equation \eqref{NLH} is said to be \textit{nondegenerate}
	if any bounded solution vanishing at infinity to the linearized equation \eqref{linear equation}  must be the linear combinations of
	the functions $\varphi_1,\cdots,\varphi_{N}$ and $\varphi_{N+1}$ defined by \eqref{null 1} and \eqref{null 2}.
\end{defi}

We now state the main result in this paper.
\begin{theo}
\label{theo2}
Let $N\geq 3$, $\lambda\in\sts{0,N}$ and $p=\frac{2N-\lambda}{N-2}$. Then the nontrivial solution $U_{c,\delta, x_0}$ of the equation \eqref{NLH} with constant sign is nondegenerate. 

\end{theo}

The result in Theorem \ref{theo2} solves an open nondegeneracy problem in \cite{GMYZ2022cvpde, MWX:Hartree}, and generalizes the partial nondegeneracy results in \cite{DY2019dcds, GWY2020na, LTX2021, MWX:Hartree} to the full rang $\lambda\in (0, N)$.  The argument in this paper is  different with those in \cite{GWY2020na,DY2019dcds} and  \cite{Lenzmann2009apde, LTX2021, Xiang2016cvpde}.  In fact, we rewrite \eqref{linear equation} as an integral form in $\R^N$, and its equivalent integral form on the sphere $\S^N$  via the stereographic projection  $\mathcal{S}: \R^N \longrightarrow \S^N$. The key observation is that together with the spherical harmonic decomposition and the Funk-Hecke formula in \cite{AH2012, DX2013book, SteinW:Fourier anal},   the weighted pushforward map $\mathcal{S}_*$ related to the stereographic projection $\mathcal{S}$ is  one-to-one map between the null space of the linearized operator and the spherical harmonic function subspace $\mathcal{H}_1^{N+1}$ of degree one. The idea with use of  the stereographic projection and  the Funk-Hecke formula is inspired by Frank and Lieb in \cite{FL2012am,FL2012},

\begin{rema}[Nondegeneracy of positive solutions to
 nonlinear elliptic equations with local nonlinearity]
There were extensive literatures to show
the nondegeneracy of positive solutions to nonlinear elliptic equations. Weinstein \cite{Weinstenin1985siam} and
Oh \cite{Oh1990cmp} made use of the spherical harmonic  expansion to 
obtain the nondegeneracy of positive solutions to
 nonlinear elliptic equation with subcritical nonlinearity
\begin{equation*}
  -{\Laplace u}\sts{x}
  +\fct{u}{x}
  -
  u^{p}\sts{x}
  =0, \quad x\in\R^N,\quad 1<p<\frac{N+2}{N-2}. 
\end{equation*}
Rey \cite{Rey1990jfa}, Dolbeault and Jankowiak \cite{DJ2014jde} made use of the stereographic projection to 
obtain the nondegeneracy of positive
bubble solutions to  nonlinear elliptic equation with critical nonlinearity
	\begin{equation}
\label{eq:nls}
-{\Laplace u}\sts{x}
-
u^{\frac{N+2}{N-2}}\sts{x}
=0, \quad x\in\R^N.
\end{equation}
 
%
%
%
%
%
%
%
%

For other applications of the spherical harmonic expansion in the nondegeneracy of positive solutions for other nonlinear elliptic equations with local nonlinearity, please refer to  \cite{Robert2017ana,GGN2013am,BWW2003cvpde, DGG2017prse,AWY2016ade,FVald:fract NLS, DPS2013pams, FrankL:ActaMath, FrankLS:CPAM, FKP2020pams,  AMW2019BSMF,PV2021prse, MW2015cmp, tao:book} and references therein.
\end{rema}
\begin{rema}[Nondegeneracy of positive solutions to
nonlinear Hartree equations] Let $N\geq 3$, $0<\lambda<N$ and $1<p<\frac{2N-\lambda}{N-2}$. 
Due to nonlocal nonlinearity of nonlinear Hartree equations,
the nondegeneracy of positive solutions to the equation
\eqref{NLH}
\begin{equation}
\label{snlh}
  -{\Laplace u}\sts{x}
  +\fct{u}{x}
  -
  \int_{\R^N}
  { \frac{ u^{p}\sts{y}}{\pabs{\,x-y\,}{\lambda}} }\diff{y}\,
  u^{p-1}\sts{x}
  =0, ~x\in\R^N, 
\end{equation}
is more subtle than
that of positive solutions to nonlinear elliptic equations with local nonlinearity.

Lenzmann firstly
made use of the multipole expansion of the Newtonian potential
, and
obtained the nondegeneracy of the ground state to
\begin{equation}
\label{3dNLH}
  -{\Laplace u}\sts{x}
  +\fct{u}{x}
  -
  \frac{1}{8\pi}\int_{\R^3}
  { \frac{ u^{2}\sts{y}}{\abs{\,x-y\,}} }\diff{y}\,
  u\sts{x}
  =0, ~x\in\R^3
\end{equation}
 in \cite{Lenzmann2009apde}, see also \cite{WW2009jmp}. For  the application of the multipole expansion of the Newtonian potential  in $\R^N (3\leq N\leq 5)$   to the nondegeneracy of positive solution to \eqref{snlh} with $\lambda=N-2$ and $p=2$, we can refer to \cite{Chen2021rm}. Later, Xiang made use of a perturbation argument to obtain the  non-degeneracy of positive solution to \eqref{snlh}
for the case that $\lambda=1$ and $p$ is slightly larger than $2$ in $\R^3$ in  \cite{Xiang2016cvpde},  and recently, Li extended the perturbation argument to the non-degeneracy of positive solution to \eqref{snlh}  for  the case that $\lambda$ close to $N-2$ and $p$ slightly larger than $2$ in $\R^N$ in \cite{Li2022} . For other applications of the perturbation argument, please refer to \cite{DY2019dcds, GWY2020na} for the case that $\lambda$ is closed to $0$ or $N$.



Recently, Li, Tang and Xu made use of the  spherical harmonic expansion and the multipole expansion of the Newtonian potential to obtain the  non-degeneracy of 
bubble solutions
to the energy-critical Hartree equation in $\R^6$
\begin{equation}
\label{6dNLH}
  -{\Laplace u}\sts{x}
  -
  \int_{\R^6}
  { \frac{ u^{2}\sts{y}}{\pabs{\,x-y\,}{4}} }\diff{y}\,
  u\sts{x}
  =0, ~x\in\R^6
\end{equation}
in \cite{LTX2021}. We can also refer to \cite{GMYZ2022cvpde}.

Recently, the authors make use of the Moser iteration method in \cite{DiMeVald:book} to obtain the $L^{\infty}(\R^N)$ regularity of the energy solution to the linearized equation \eqref{linear equation}, and  show that the nontrivial solution $U_{c,\delta, x_0}$ of the equation \eqref{NLH} with constant sign is nondegenerate in the energy space $\dot H^1(\R^N)$ in \cite{LLTX2023}. At the same time, the authors consider  long time dynamics of the radial threshold solution  of  the equation \eqref{gNLH} and its rigidity classification in \cite{LLTX2023}, which depends on  the nondegeneracy of the bubble solutions in the energy space $\dot H^1(\R^N)$ and the spectrum of the linearized operator.
\end{rema}

\begin{rema}[Application of nondegeneracy of positive solution in the  construction of multi-bubble solutions]
	The existence of bubble solutions to the energy-critical nonlinear Hartree equation
	has
	been well-studied recently, see \cite{GWY2020na,DY2019dcds,GHPS2019cvpde,GY2018, MWX:Hartree}.
It is interesting to construct
	the existence of multi-bubble solutions to the
	equation \eqref{NLH}.
	To the best of our knowledge, there are few results concerning
	multi-bubble solutions to \eqref{NLH} except that in \cite{GMYZ2022cvpde}.
	However, for the limiting case $\lambda=0$ in \eqref{NLH}, i.e.
	\begin{equation}
	\label{eq:nls}
	-{\Laplace u}\sts{x}
	-
	u^{\frac{N+2}{N-2}}\sts{x}
	=0, \quad x\in\R^N,
	\end{equation}
	the multi-bubble solutions has been
	constructed by using the \LS argument in \cite{PMPP2013SNS,PMPP2011jde}, where
	the nondegeneracy of positive solutions plays a crucial role. 
\end{rema}

Lastly, the rest of this paper is organized as follows. 
In \Cref{sec:preliminary},
we introduce some notation, the preliminary results about the stererographic projection and the Funk-Hecke formula of the spherical harmonic functions.
In \Cref{sec:proof of theo2}, we prove \Cref{theo2}. 

\noindent \subsection*{Acknowledgements.}
The authors were supported by National Key Research and Development Program of China (No. 2020YFA0712900) and by NSFC (No. 12371240, No. 12431008). 

\section{Notation and Preliminary Results}
\label{sec:preliminary}
In this section, we  introduce
some notation. We denote
$
\jp{x}=\sts{1+\pabs{x}{2}}^{\frac{1}{2}},
$
and use $\S^N$ to denote the unit sphere in $\R^{N+1}$,
 i.e.
\begin{equation*}
  \S^N=\ltl{\xi=\sts{\xi_1,\xi_2,\cdots,\xi_{N+1}}\in\R^{N+1}
  \quad \middle\vert\quad  \sum_{j=1}^{N+1}\xi_j^2=1},
\end{equation*}
and $g_{ij}\sts{1\leq i,\, j\leq N+1}$ stand for the metric on $\S^N$, which is inherited from $\R^{N+1}$. For any $1\leq p<+\infty$, let us denote by
$L^p\sts{\R^N}$ and $L^p\sts{\S^N}$
the space of real-valued $p$-th power integrable functions on
$\R^N$ and $\S^n$. Moreover, with a little abuse of notation,
we equip $L^p\sts{\R^N}$ and $L^p\sts{\S^N}$ with the norms:
\begin{equation*}
\norm{f}{p}=\sts{\int_{\R^N}\pabs{\fct{f}{x}}{p}\diff{x}}^{\frac{1}{p}},
\qtq{~for~} f\in L^p\sts{\R^N},
\end{equation*}
and
\begin{equation*}
\norm{F}{p}=\sts{\int_{\S^N}\pabs{\fct{F}{\xi}}{p}\diff{\xi}}^{\frac{1}{p}},
\qtq{~for~} F\in L^p\sts{\S^N},
\end{equation*}
where $\diff{\xi}$ is the standard volume element on the sphere $\S^N$. 

We denote the stereographic projection $ \Scal: \R^N  \mapsto  \S^N\setminus\ltl{\sts{0,0,\cdots,0,-1}}$ by 
\begin{align*}
  \Scal x =  \sts{
    \frac{2x}{1+\pabs{x}{2}},
    \frac{1-\pabs{x}{2}}{1+\pabs{x}{2}}
  },
\end{align*}
and  its inverse map $  \Scal^{-1}:  \S^N\setminus\ltl{\sts{0,0,\cdots,0,-1}}  \mapsto  \R^N$ by
\begin{align*}
 \Scal^{-1} \sts{\xi_1,\xi_2,\cdots,\xi_{N+1}}
=
  \sts{
    \frac{\xi_1}{1+\xi_{N+1}},
    \frac{\xi_2}{1+\xi_{N+1}},
    \cdots,
    \frac{\xi_N}{1+\xi_{N+1}}
  }.
\end{align*}

Let $
\label{rho}
\fct{\rho}{x}=
\sts{\frac{2}{1+\pabs{x}{2}}}^{\frac{1}{2}}
$.  From  \cite{lieb2001analysis, FL2012},  we have
\begin{align}
\label{stere proj:identity}
  g_{ij} = \fct{\rho^{4}}{x}\delta_{ij}, \;\;   {\abs{\Scal x-\Scal y}}
=
{\abs{ x- y}}
\fct{\rho}{x}\fct{\rho}{y}
,  
\end{align}
and 
\begin{align}\label{trans:det}
\diff{\xi}=\fct{\rho^{2N}}{x}\diff{x}.
\end{align}
 Therefore, for any $F\in L^1\sts{\S^N}$, we have the following identity 
 \begin{equation*}
 \sint \fct{F}{\xi}\diff{\xi}
 =\rint \fct{F}{\Scal x}{\fct{\rho^{2N}}{x}}\diff{x}.
 \end{equation*}
%

In order to relate the functions between in $\R^N$ and on the sphere $\S^N$,  we can compose the functions in $\R^N$ and $\S^N$ with the maps $\mathcal{S}^{\pm}$.  For any $f:\R^N\mapsto \R$,  we denote  the weighted pushforward map $  \Scal_{\ast}f:
 \S^N\setminus\ltl{\sts{0,0,\cdots,0,-1}}  \mapsto  \R $ by 
\begin{align}
\label{R2S under S}
 \Scal_{\ast}f( \xi) =
  \fct{\rho^{2-N}}{\Scal^{-1}\xi}\fct{f}{\Scal^{-1}\xi},
\end{align}
and for any
$F:\S^N\setminus\ltl{\sts{0,0,\cdots,0,-1}}\mapsto \R$,
we denote the weighted pullback map $  \Scal^{\ast}F:
 \R^N  \mapsto  \R$ by 
\begin{align}
\label{S2R under S}
  \Scal^{\ast}F (x)=
  \fct{\rho^{N-2}}{x}\fct{F}{\Scal x}.
\end{align}

A simple calculation shows that
\begin{prop}\label{prop:Rn to Sn}
	Let $\varphi_j$, $1\leq j\leq N+1$ be defined by \eqref{null 1} and \eqref{null 2} and $\Scal_{*}$ be defined by \eqref{R2S under S}, then for any $\xi\in \S^N$ and $1\leq j \leq N$,  we have
\begin{align}
\label{S:phi}
	\Scal_{*} \varphi_j (\xi) =  \frac{2-N}{2}\, 2^{\frac{2-N}{2}} \xi_j, \;\; \text{and}\;\;
		\Scal_{*} \varphi_{N+1} (\xi)  = \frac{N-2}{2}\, 2^{\frac{2-N}{2}} \xi_{N+1}.
\end{align}
\end{prop}
\begin{proof}
First of all, for $j=1,2,\cdots,N$, we have by \eqref{R2S under S} that 
\begin{align}
\label{eq:112601}
\Scal_{*} \varphi_j \left(\Scal x\right) =  \fct{\rho^{2-N}}{x}\fct{\varphi_j}{x}, 
\end{align}
which, together with \eqref{null 1} and \eqref{null 2}, implies that 
\begin{align*}
\Scal_{*} \varphi_j \left(\Scal x\right) =  
\frac{2-N}{2}2^{\frac{2-N}{2}}\frac{2x_j}{1+\abs{x}^2}=
\frac{2-N}{2}2^{\frac{2-N}{2}}\left(\Scal x\right)_j=
\frac{2-N}{2}2^{\frac{2-N}{2}}\xi_j,
\end{align*}
and
\begin{align*}
\Scal_{*} \varphi_{N+1} \left(\Scal x\right) =  
\frac{N-2}{2}2^{\frac{2-N}{2}}\frac{1-\abs{x}^2}{1+\abs{x}^2}=
\frac{N-2}{2}2^{\frac{2-N}{2}}\left(\Scal x\right)_{N+1}=
\frac{N-2}{2}2^{\frac{2-N}{2}}\xi_{N+1},
\end{align*}
which implies \eqref{S:phi} and completes the proof.
\end{proof}

 We now introduce the spherical harmonic functions, which are  related to the spectral properties of the Laplace-Beltrami opertor on the sphere $\S^N$ 
 (see  \cite{AH2012, DX2013book, SteinW:Fourier anal}).
 In fact, we have the following orthogonal decomposition:
 \begin{equation}
 \label{orthogonal decomposition:L2}
 L^2\sts{\S^N}=\bigoplus\limits_{k=0}^{\infty}\Hscr_{k}^{N+1},
 \end{equation}
 where $\Hscr_{k}^{N+1}\sts{k\geq 0}$ denote
 the mutually orthogonal subspace of
 the restriction on $\S^N$ of real,  homogeneous harmonic polynomials of degree $k$, and 
 \begin{equation*}\label{dim H}
 \dim{\Hscr_{k}^{N+1}}
 :=
 \begin{dcases}
 1, & \mbox{if } k=0, \\
 N+1, & \mbox{if } k=1, \\
 \binom{k+N}{k}-\binom{k-2+N}{k-2}, & \mbox{if } k\geq 2.
 \end{dcases}
 \end{equation*}
 We will use
 $\ltl{Y_{k,j}\mid 1\leq j\leq \dim{\Hscr_{k}^{N+1}}}$
 to denote an orthonormal basis of $\Hscr_{k}^{N+1}$.
 In particular, we have 
 \begin{equation*}
 \fct{Y_{1,j}}{\xi}=
 \sqrt{\frac{\sts{N+1}
 		\fct{\Gamma}{\frac{N}{2}}}{ 2\pi^{\frac{N}{2}}}}\xi_j,
 \quad 1\leq j\leq N+1,
 \end{equation*}
 and
 \begin{equation}\label{basis for eig 1}
 {\Hscr_{1}^{N+1}}
 =
 \mathrm{span}
 \ltl{
 	\xi_j
 	\,\middle\vert\,
 	1\leq j\leq N+1
 },
 \end{equation}
 which together with Proposition \ref{prop:Rn to Sn} implies that the weighted pushfoward map $\Scal_{*}$ is one-to-one map from the subspace $\text{Span}\{\varphi_j,   1\leq j\leq N+1\}\subset L^{\infty}(\R^N)$ to the subspace $ {\Hscr_{1}^{N+1}}$, and so is the weighted pullback map $\Scal^{*}:  {\Hscr_{1}^{N+1}}\rightarrow  \text{Span}\{\varphi_j,   1\leq j\leq N+1\}$. This is the key observation in the proof of Theorem \ref{theo2}. 
 
  
We recall the Funk-Hecke formula of the spherical harmonic functions as follows.
 
 \begin{lemm}[\cite{AH2012, DX2013book}]
 	\label{lem:Funk-Hecke}
 	Let $\lambda\in\sts{0,N}$, the integer $k\geq 0$ and $\fct{\mu_k}{\lambda}$ be defined by
 	\eqref{mu k lambda}, then for any $Y \in \Hscr_{k}^{N+1}$, we have
 	\begin{equation}
 	\label{Funk-Hecke}
 	\sint \frac{1}{\pabs{\xi-\eta}{\lambda}}\fct{Y}{\eta}\diff{\eta}
 	=\fct{\mu_k}{\lambda}\fct{Y}{\xi}. 
 	\end{equation}
 	where 
 		\begin{equation}\label{mu k lambda}
 	\fct{\mu_k}{\lambda} = 2^{N-\lambda}\pi^{\frac{N}{2}}
 	\frac{
 		\fct{\Gamma}{k+\frac{\lambda}{2}}
 		\fct{\Gamma}{\frac{N-\lambda}{2}}
 	}{
 		\fct{\Gamma}{\frac{\lambda}{2}}
 		\fct{\Gamma}{k+N-\frac{\lambda}{2}}
 	}.
 	\end{equation}
 \end{lemm}

 In particular, the simple calculation gives that
\begin{gather}
\label{mu01 N-2}
\fct{\mu_0}{N-2}
=
\frac{8}{N}\cdot \frac{\pi^{\frac{N}{2}}}{\fct{\Gamma}{\frac{N}{2}}},
\qquad
\fct{\mu_1}{N-2}
=
\frac{N-2}{N+2}\cdot \fct{\mu_0}{N-2},
\\
\label{mu01 lambda}
\fct{\mu_0}{\lambda}
=
2^{N-\lambda}\pi^{\frac{N}{2}}\cdot 
\frac{
	\fct{\Gamma}{\frac{N-\lambda}{2}}
}{
	\fct{\Gamma}{N-\frac{\lambda}{2}}
},
\qquad
\fct{\mu_1}{\lambda} =\frac{\lambda}{2N-\lambda} \cdot \fct{\mu_0}{\lambda}.
\end{gather}
and 
\begin{align}
\label{monotone mu k}
&\fct{\mu_k}{\lambda}>\fct{\mu_{k+1}}{\lambda},\qtq{for all} k\geq 0,
\end{align}

 As a direct consequence of the  Funk-Hecke formula,
 we have 
 \begin{lemm}[]
 	\label{double funk-hecke}
 	Let $\lambda\in\sts{0,N}$, the integer $k\geq 0$ and $\fct{\mu_k}{\lambda}$ be defined by
 	\eqref{mu k lambda}, then for any $Y \in \Hscr_{k}^{N+1}$,
 	we have
 \begin{align}
 \label{DFH:1}
 	\sint\sint
 	\frac{1}{\pabs{\xi-\eta}{N-2}}
 	\frac{ 1 }{\pabs{\eta-\sigma}{\lambda}}
 	\fct{Y}{\sigma}\diff{\sigma}
 	\diff{\eta}
 	=\fct{\mu_k}{N-2}\fct{\mu_k}{\lambda}\fct{Y}{\xi}
 	,
 	\\
 \label{DFH:2}
 	\sint\sint \frac{1}{\pabs{\xi-\eta}{N-2}}
 	\frac{1}{\pabs{\eta-\sigma}{\lambda}}\fct{Y}{\eta}
 	\diff{\eta}\diff{\sigma}
 	=
 	\fct{\mu_k}{N-2}\fct{\mu_0}{\lambda}\fct{Y}{\xi}.
 \end{align}
 \end{lemm}
\begin{proof}
We first show \eqref{DFH:1}. Indeed, by Lemma \ref{lem:Funk-Hecke}, we have, for any $Y \in \Hscr_{k}^{N+1}$,
\begin{align}
\label{eq:112605}
\sint \frac{1}{\pabs{\eta-\sigma}{\lambda}}\fct{Y}{\sigma}\diff{\sigma}
=\fct{\mu_k}{\lambda}\fct{Y}{\eta}.
\end{align}
by Lemma \ref{lem:Funk-Hecke} with $\lambda=N-2$ again, 
 we get
\begin{align*}
\sint \frac{1}{\pabs{\xi-\eta}{N-2}}
\left(\sint \frac{1}{\pabs{\eta-\sigma}{\lambda}}\fct{Y}{\sigma}\diff{\sigma} \right)
\diff{\eta}=\fct{\mu_k}{N-2}\fct{\mu_k}{\lambda}\fct{Y}{\xi},
\end{align*}
which, together with Fubini's theorem, implies \eqref{DFH:1}. 

Next, we show \eqref{DFH:2}. By Lemma \ref{lem:Funk-Hecke} and the fact that $1\in \Hscr_{0}^{N+1}$, we have
\begin{align}
\label{eq:112606}
\sint \frac{1}{\pabs{\eta-\sigma}{\lambda}}\diff{\sigma}
	=\fct{\mu_0}{\lambda}.
\end{align}
By Lemma \ref{lem:Funk-Hecke} with  $\lambda=N-2$ again, we obtain, for any $Y \in \Hscr_{k}^{N+1}$,
\begin{align*}
\sint \frac{1}{\pabs{\xi-\eta}{N-2}}\fct{Y}{\eta}
\left(
\sint \frac{1}{\pabs{\eta-\sigma}{\lambda}}\diff{\sigma} 
\right)
\diff{\eta}=\fct{\mu_k}{N-2}\fct{\mu_0}{\lambda}\fct{Y}{\xi},
\end{align*}
which implies \eqref{DFH:2}, and completes the proof. 
\end{proof}

\section{Proof of \Cref{theo2}}
\label{sec:proof of theo2}
In this section, we will prove \Cref{theo2}, which is the main result in this paper. 
We firstly give an integral estimates, which will be used in next decay estimate.
\begin{lemm}
	\label{lem:decay}
	Let $\lambda\in\sts{0,N}$ and $\theta+\lambda>N$. Then
	\begin{equation}
	\label{decay all}
	\rint\frac{1}{\pabs{x-y}{\lambda}}\frac{1}{\jp{y}^{\theta}}\diff{y}
	\lesssim
	\begin{dcases}
	\jp{x}^{N-\lambda-\theta}, & \qtq{if } \theta<N, \\
	\jp{x}^{-\lambda}\sts{1+{\log}{\jp{x}}}, & \qtq{if } \theta=N, \\
	\jp{x}^{-\lambda}, & \qtq{if } \theta>N.
	\end{dcases}
	\end{equation}
\end{lemm}
\begin{proof}
	The proof will be divided into three cases.
	
	\textbf{Case 1:~$\bm{x=0}$.}
	By a direct computation,
	using the assumptions that $N>\lambda$ and $N<\theta+\lambda$,
	we have
	\begin{align}
	\label{case1}
	\rint\frac{1}{\pabs{y}{\lambda}}\frac{1}{\jp{y}^{\theta}}\diff{y}
	\lesssim 1.
	\end{align}
	
	\textbf{Case 2:~$\bm{x\in\mathrm{B}\sts{0,1}\setminus\ltl{0}}$.}
	On the one hand, for $y\in \mathrm{B}\sts{x,2\abs{x}}$, we have
	$\jp{y}\approx  1$, which together with $\abs{x}<1$ and $N>\lambda$
	implies that,
	\begin{align}
	\label{case2:1}
	\int_{\mathrm{B}\sts{x,2\abs{x}}}
	\frac{1}{\pabs{x-y}{\lambda}\jp{y}^{\theta}}\diff{y}
	\lesssim
	\int_{0}^{2}r^{N-\lambda-1}\diff{r}
	\lesssim 1.
	\end{align}
	On the other hand, for $y\in \R^N\setminus\mathrm{B}\sts{x,2\abs{x}}$,
	we have $\abs{y}\approx \abs{x-y}$,
	which, together with the assumption that $N<\lambda+\theta$, implies that
	\begin{align}
	\label{case2:2}
	\int_{\R^N\setminus \mathrm{B}\sts{x,2\abs{x}}}
	\frac{1}{\pabs{x-y}{\lambda}\jp{y}^{\theta}}\diff{y}
	\lesssim
	\int_{\R^N\setminus \mathrm{B}\sts{0,2\abs{x}}}
	\frac{1}{\pabs{y}{\lambda}\jp{y}^{\theta}}\diff{y}
	\lesssim \jp{x}^{N-\lambda-\theta}.
	\end{align}
	Combining \eqref{case2:1} with \eqref{case2:2},
	we get
	\begin{align}\label{case2}
	\rint\frac{1}{\pabs{x-y}{\lambda}\jp{y}^{\theta}}\diff{y}
	\lesssim 1,\qtq{for any}
	x\in\mathrm{B}\sts{0,1}\setminus\ltl{0}.
	\end{align}
	
	\textbf{Case 3:~$\bm{x\in\R^N\setminus\mathrm{B}\sts{0,1}}$.}
	Firstly, noticing that for any
	$y\in \mathrm{B}\sts{x,\frac{\abs{x}}{2}}$,
	we have $\jp{y}\approx \jp{x}$.
	Hence,
	\begin{equation}
	\label{case3:1}
	\int_{\mathrm{B}\sts{x,\frac{\abs{x}}{2}}}
	\frac{1}{\pabs{x-y}{\lambda}\jp{y}^{\theta}}\diff{y}
	\lesssim \frac{1}{\jp{x}^{\theta}}
	\int_{0}^{\frac{\abs{x}}{2}}r^{N-\lambda-1}\diff{r}
	\lesssim
	\jp{x}^{N-\lambda-\theta},
	\end{equation}
	where we used the assumption that $N>\lambda$.
	
	Secondly, for any
	$y\in \mathrm{B}\sts{x,2{\abs{x}}}
	\setminus \mathrm{B}\sts{x,\frac{\abs{x}}{2}}$,
	we have
	\begin{equation*}
	\abs{y}\leq 3\abs{x},\qtq{and} \abs{x-y}\approx \jp{x},
	\end{equation*}
	which implies that,
	\begin{align}\label{case3:2}
	\int_{
		\mathrm{B}\sts{x,2{\abs{x}}}
		\setminus
		\mathrm{B}\sts{x,\frac{\abs{x}}{2}}
	}
	\frac{1}{\pabs{x-y}{\lambda}\jp{y}^{\theta}}\diff{y}
	\lesssim
	\frac{1}{\jp{x}^{\lambda}}
	\int_{0}^{3\abs{x}}\frac{r^{N-1}}{\jp{r}^{\theta}}\diff{r}
	\lesssim
	\begin{dcases}
	\jp{x}^{N-\lambda-\theta}, & \text{if}\; \theta<N, \\
	\jp{x}^{-\lambda}\sts{1+{\log}{\jp{x}}}, & \text{if}\; \theta=N, \\
	\jp{x}^{-\lambda}, & \text{if}\; \theta>N.
	\end{dcases}
	\end{align}
	
	Thirdly, for any $y\in\R^N\setminus\mathrm{B}\sts{x,2{\abs{x}}}$,
	we have $\jp{y}\approx \abs{y-x} $, which implies that,
	\begin{align}
	\label{case3:3}
	\int_{
		\R^N\setminus\mathrm{B}\sts{x,2{\abs{x}}}
	}
	\frac{1}{\pabs{x-y}{\lambda}\jp{y}^{\theta}}\diff{y}
	\lesssim
	\int_{
		\R^N\setminus\mathrm{B}\sts{x,2{\abs{x}}}
	}
	\frac{1}{\pabs{x-y}{\lambda+\theta}\diff{y}}
	\lesssim
	\pabs{x}{N-\lambda-\theta}
	\approx
	\jp{x}^{N-\lambda-\theta},
	\end{align}
	where we used the assumption that $N<\lambda+\theta$.
	
	Finally, using \eqref{case3:1}, \eqref{case3:2} and
	\eqref{case3:3}, we obtain that,
	for any $x\in\R^N\setminus\mathrm{B}\sts{0,1}$,
	\begin{align}
	\label{case3}
	\int_{
		\R^N
	}
	\frac{1}{\pabs{x-y}{\lambda}\jp{y}^{\theta}}\diff{y}
	\lesssim
	\begin{dcases}
	\jp{x}^{N-\lambda-\theta}, & \text{if } \theta<N, \\
	\jp{x}^{-\lambda}\sts{1+{\log}{\jp{x}}}, & \text{if } \theta=N, \\
	\jp{x}^{-\lambda}, & \text{if } \theta>N.
	\end{dcases}
	\end{align}
	
Combining \eqref{case1}, \eqref{case2} with \eqref{case3},
	we obtain\eqref{decay all} and complete the proof.
\end{proof}

Next, we denote  
\begin{align}
\fct{\Npzc\sts{\varphi}}{x}
& =
p\rint
\frac{ \fct{u^{p-1}}{y}\fct{\varphi}{y}
}{\pabs{x-y}{\lambda}}\diff{y}\,
\fct{u^{p-1}}{x}+ 
\sts{p-1}\rint
\frac{ \fct{u^{p}}{y}
}{\pabs{x-y}{\lambda}}\diff{y}\,
\fct{u^{p-2}}{x}\fct{\varphi}{x}  \notag\\
& =: 
\fct{\Npzc_{1}\sts{\varphi}}{x}
+
\fct{\Npzc_{2}\sts{\varphi}}{x}, \label{N}
\end{align}
where $u$ is  defined by \eqref{u}. 

\begin{lemm}
\label{lem:N:decay}
Let $\lambda\in \sts{0,N}$, and $\theta\in[0,N-2]$.
If $\varphi$ satisfies $\abs{\fct{\varphi}{x}}\lesssim \frac{1}{\jp{x}^{\theta}}$,
then we have
\begin{equation}
\label{N:decay}
  \abs{\fct{\Npzc\sts{\varphi}}{x}}\lesssim
\frac{1}{\jp{x}^{\theta+4}}.
\end{equation}
\end{lemm}

\begin{proof}
First, by \Cref{lem:decay}, we have
\begin{align}\label{N1:decay original}
\abs{\fct{\Npzc_1\sts{\varphi}}{x}}
\lesssim
\frac{1}{\jp{x}^{N+2-\lambda}}\rint \frac{1}{\pabs{x-y}{\lambda}}
\frac{1}{\jp{y}^{N+2-\lambda}}
\frac{1}{\jp{y}^{\theta}}\diff{y}
\lesssim
\begin{dcases}
\frac{1}{\jp{x}^{N+2}}, & \qtq{if} \theta+2>\lambda,\\
\frac{1+\log\jp{x}}{\jp{x}^{N+2}}, & \qtq{if} \theta+2=\lambda,\\
\frac{1}{\jp{x}^{N+\theta+4-\lambda}}, & \qtq{if} \theta+2<\lambda.
\end{dcases}
\end{align}
which, together with  the fact that $N+2\geq\theta +4$, $N>\lambda$, and $\frac{1+\log\jp{x}}{\jp{x}^{N-\lambda}}\lesssim 1$, implies that
\begin{equation}\label{N1:decay}
  \abs{\fct{\Npzc_1\sts{\varphi}}{x}}
\lesssim \frac{1}{\jp{x}^{\theta +4}}.
\end{equation}

Next, by \Cref{lem:decay} once again, we have
\begin{align}
\label{N2:decay}
  \abs{\fct{\Npzc_2\sts{\varphi}}{x}}
\lesssim
\int_{\R^N} \frac{1}{\pabs{x-y}{\lambda}}\frac{1}{\jp{y}^{2N-\lambda}}\diff{y}
\frac{1}{\jp{x}^{4+\theta-\lambda}}
\lesssim
\frac{1}{\jp{x}^{\theta+4}}.
\end{align}

By combining \eqref{N1:decay} with \eqref{N2:decay}, we can obtain
the result.
\end{proof}

By \Cref{lem:N:decay} and the classical Riesz potential theory in \cite{CLO2006, Stein:Sing Integ}, 
we can rewrite the linearized equation \eqref{linear equation} as an integral form.

\begin{lemm}
\label{lem:lin eq to IE}
Let $N\geq 3$, $\lambda\in \sts{0,N}$.
If $\varphi\in L^{\infty}(\R^N)$ satisfies \eqref{linear equation},
then we have
\begin{equation}\label{IE for linear equation}
  \fct{\varphi}{x}=
\Const{G}\sts{N}{\bm\alpha}\sts{N,\lambda}
\rint\frac{1}{\pabs{x-y}{N-2}}\fct{\Npzc}{\varphi}\sts{y}\diff{y},
\end{equation}
where $\fct{\Npzc}{\varphi}$ is given by \eqref{N} and  the constant
$
\Const{G}\sts{N}=
\frac{\fct{\Gamma}{\frac{N}{2}}}{2\sts{N-2}\pi^{\frac{N}{2}}}.
$
Moreover, we have the following improved decay estimate
\begin{equation}
\label{varphi:decay}
  \abs{\fct{\varphi}{x}}\lesssim \frac{1}{\jp{x}^{N-2}}.
\end{equation}
\end{lemm}

\begin{proof}
 \eqref{IE for linear equation} follows from \Cref{lem:N:decay} and the classical Riesz potential theory in \cite{CLO2006, Stein:Sing Integ}.
Therefore, it suffices to show the decay estimate \eqref{varphi:decay}.
It follows from \eqref{IE for linear equation} and the bootstrap argument on the decay rate.

Since $\varphi\in L^{\infty}(\R^N)$,
we have $\abs{\fct{\varphi}{x}}\lesssim1,$
which together with \Cref{lem:N:decay} implies that
$
  \abs{\fct{\Npzc}{\varphi}\sts{x}}\lesssim \frac{1}{\jp{x}^4}.
$
Therefore, by \Cref{lem:decay}, we have for some $0<\epsilon\ll 1$ that
\begin{equation*}
  \abs{\fct{\varphi}{x}}\lesssim
  \begin{cases}
    \frac{1}{\jp{x}^{N-2}}, & \mbox{if~~} N < 4, \\
    \frac{1}{\jp{x}^{2(1-\epsilon)}}, & \mbox{if~~} N\geq 4.
  \end{cases}
\end{equation*}
which reduces to boost the case $N\geq 4$.
Now, we assume by induction for some $1\leq j\leq \left[\frac{N-2}{2}\right]$, which is the integer part of $\frac{N-2}{2}$ that
$$\abs{\fct{\varphi}{x}}\lesssim \frac{1}{\jp{x}^{2(j-\epsilon)}}.$$
Repeating a similar argument as above, we can obtain that $\abs{\fct{\Npzc}{\varphi}\sts{x}}\lesssim \frac{1}{\jp{x}^{2j+4-\epsilon}}$ and 

\begin{equation*}
\abs{\fct{\varphi}{x}}\lesssim
\begin{cases}
\frac{1}{\jp{x}^{N-2}}, & \mbox{if~~} N < 2j+4-\epsilon, \\
\frac{1}{\jp{x}^{2(j+1-\epsilon)}}, & \mbox{if~~} N\geq 2j+4-\epsilon.
\end{cases}
\end{equation*}
By the induction argument for $j=1, \cdots,  \left[\frac{N-2}{2}\right]$, we can obtain the result.
%
%
%
%
%
\end{proof}

\begin{rema}The decay estimate in the above lemma is optimal since the null element $\varphi_{N+1}$ defined by \eqref{null 2} satisfies  \eqref{varphi:decay}. 
	\end{rema}

Now we will make use of the stereographic projection to transform the integral equation \eqref{IE for linear equation} on $\R^N$ to that on the sphere $\S^N$. Let us denote 
\begin{equation}
\label{T}
\fct{\Tcal_{\S^N}{\Phi}}{\xi}
=
p\fct{\Tcal_{\S^N,1}\Phi}{\xi}+
\sts{p-1}\fct{\Tcal_{\S^N,2}\Phi}{\xi},
\end{equation}
where $\xi\in \S^N$ and 
\begin{align}
\label{T1}
\fct{\Tcal_{\S^N,1}\Phi}{\xi}
&=2^{-(p-1)(N-2)}
\sint\sint
\frac{1}{\pabs{\xi-\eta}{N-2}}
\frac{ 1 }{\pabs{\eta-\sigma}{\lambda}}
\fct{\Phi}{\sigma}\diff{\sigma}
\diff{\eta},
\\
\label{T2}
\fct{\Tcal_{\S^N,2}\Phi}{\xi}
&=2^{-(p-1)(N-2)}
\sint\sint \frac{1}{\pabs{\xi-\eta}{N-2}}
\frac{1}{\pabs{\eta-\sigma}{\lambda}}\fct{\Phi}{\eta}
\diff{\eta}\diff{\sigma}.
\end{align}

\begin{lemm}
\label{lem:IE sphere}
Let $N\geq 3$, and $\lambda\in \sts{0,N}$.
If $\varphi$ satisfies \eqref{IE for linear equation}
with
$\abs{\fct{\varphi}{x}}\lesssim \frac{1}{\jp{x}^{N-2}}$,
then
${\Scal_{\ast}\varphi}\in L^2\sts{\S^N}$ satisfies
\begin{equation}
\label{IE sphere}
\fct{{\Scal_{\ast}\varphi}}{\xi}
=
\Const{G}\sts{N}{\bm\alpha}\sts{N,\lambda}
\fct{\Tcal_{\S^N}{{\Scal_{\ast}\varphi}}}{\xi}.
\end{equation}
\end{lemm}
\begin{proof}
By \eqref{trans:det}, \eqref{R2S under S} and the estimate $\abs{\fct{\varphi}{x}}\lesssim \frac{1}{\jp{x}^{N-2}}$,
we have
\begin{equation*}
\sint\abs{ \fct{{\Scal_{\ast}\varphi}}{\xi} }^2\diff{\xi}=   \rint \abs{ \fct{\varphi}{ x} }^2 \fct{\rho^{4}}{x}\diff{x}
  \lesssim
  \rint \frac{1}{\jp{x}^{2N}}\diff{x}<+\infty. 
\end{equation*}
 
Now we turn to the proof of \eqref{IE sphere}. 
First, by the definition of \eqref{N}  with \eqref{u}, and \eqref{stere proj:identity},
we have
\begin{align*}
    \fct{\Npzc_1\sts{\varphi}}{x}
=&p\, 2^{-(p-1)(N-2)}\fct{\rho^{N+2}}{x}
\rint
\frac{ 1 }{\pabs{\Scal x-\Scal y}{\lambda}}
\left[\fct{\rho^{2-N}}{y}\fct{\varphi}{y}\right]\fct{\rho^{2N}}{y}\diff{y},
\end{align*}
which together with \eqref{R2S under S}
implies that
\begin{align}
\label{N1 on sphere}
    \fct{\Npzc_1\sts{\varphi}}{x}
=& p\, 2^{-(p-1)(N-2)}
\fct{\rho^{N+2}}{x}
\sint
\frac{ 1 }{\pabs{\Scal x-\eta}{\lambda}}
\fct{{\Scal_{\ast}\varphi}}{\eta}\diff{\eta}.
\end{align}
By \eqref{stere proj:identity} and \eqref{N1 on sphere}, we have
\begin{align}
\notag
&
\rint \frac{1}{\pabs{x-y}{N-2}}\fct{\Npzc_1\sts{\varphi}}{y}\diff{y}
\\
\label{newton N1:1}
=&p\, 2^{-(p-1)(N-2)}
\fct{\rho^{N-2}}{x}
\rint
\frac{1}{\pabs{\Scal x-\Scal y}{N-2}}
\sint
\frac{ 1 }{\pabs{\Scal y-\sigma}{\lambda}}
\fct{{\Scal_{\ast}\varphi}}{\sigma}\diff{\sigma}
\fct{\rho^{2N}}{y}\diff{y},
\end{align}
which together with the stereographic projection implies that
\begin{align}
\notag
 \rint \frac{1}{\pabs{x-y}{N-2}} & \fct{\Npzc_1\sts{\varphi}}{y}\diff{y} \\
=&p\, 2^{-(p-1)(N-2)}
\fct{\rho^{N-2}}{x}
\sint\sint
\frac{1}{\pabs{\Scal x-\eta}{N-2}}
\frac{ 1 }{\pabs{\eta-\sigma}{\lambda}}
\fct{{\Scal_{\ast}\varphi}}{\sigma}\diff{\sigma}
\diff{\eta}.\label{newton N1}
\end{align}

Next, by  the definition of  \eqref{N} with
\eqref{u} and \eqref{stere proj:identity}, 
we  have
\begin{align*}
    \fct{\Npzc_2\sts{\varphi}}{x}
=& (p-1)\, 2^{-(p-1)(N-2)}
\fct{\rho^{4}}{x}\fct{\varphi}{x}
\rint
\frac{1}{\pabs{\Scal x-\Scal y}{\lambda}}
\fct{\rho^{2N}}{y}\diff{y},
\end{align*}
which together with the stereographic projection implies that
\begin{align}
\label{N2 on sphere}
    \fct{\Npzc_2\sts{\varphi}}{x}
= (p-1)\, 2^{-(p-1)(N-2)}
\fct{\rho^{4}}{x}\fct{\varphi}{x}
\sint
\frac{1}{\pabs{\Scal x-\eta}{\lambda}}\diff{\eta}.
\end{align}
By  \eqref{stere proj:identity}  and \eqref{N2 on sphere},
we have
\begin{align}
\notag
&
\rint \frac{1}{\pabs{x-y}{N-2}}\fct{\Npzc_2\sts{\varphi}}{y}\diff{y}
\\
\label{newton N2:1}
=& (p-1)\, 2^{-(p-1)(N-2)}
\fct{\rho^{N-2}}{x}
\rint \frac{1}{\pabs{\Scal x-\Scal y}{N-2}}
\sint
\frac{1}{\pabs{\Scal y-\sigma}{\lambda}}\diff{\sigma}
\,\fct{\rho^{2-N}}{y}\fct{\varphi}{y}\fct{\rho^{2N}}{y}\diff{y}.
\end{align}
By the stereographic projection and \eqref{R2S under S},
we have
\begin{align}
\notag
&
\rint \frac{1}{\pabs{x-y}{N-2}}\fct{\Npzc_2\sts{\varphi}}{y}\diff{y}
\\
\label{newton N2}
=& (p-1)\, 2^{-(p-1)(N-2)}
\fct{\rho^{N-2}}{x}
\sint \frac{1}{\pabs{\Scal x-\eta}{N-2}}
\fct{{\Scal_{\ast}\varphi}}{\eta}
\sint
\frac{1}{\pabs{\eta-\sigma}{\lambda}}
\diff{\eta}\diff{\sigma}.
\end{align}

Inserting  \eqref{newton N1} and \eqref{newton N2} into \eqref{IE for linear equation}, we obtain
\begin{equation*}
  \fct{\rho^{2-N}}{x}\fct{\varphi}{x}
=
\Const{G}\sts{N}{\bm\alpha}\sts{N,\lambda}
\fct{\Tcal_{\S^N}{{\Scal_{\ast}\varphi}}}{\Scal x},
\end{equation*}
which together with the fact \eqref{R2S under S} 
implies the result. 
\end{proof}

Now we can use the spherical harmonic decomposition and the Funk-Hecke formula of the spherical harmonic functions in \cite{AH2012, DX2013book, SteinW:Fourier anal} to classify the solution of the equation  \eqref{IE sphere} on the sphere $\S^N $.

\begin{prop}
\label{lem:sltn to IE sphere}
Let $N\geq 3$, $\lambda\in \sts{0,N}$, and $\Tcal_{\S^N}$
be defined by \eqref{T}.
If $\Phi\in L^2\sts{\S^N}\setminus\ltl{0}$ satisfies
\begin{equation}
\label{IE sphere Phi}
  \fct{\Phi}{\xi}=
  \Const{G}\sts{N}{\bm\alpha}\sts{N,\lambda}
  \fct{\Tcal_{\S^N}\Phi}{\xi},
\end{equation}
then $\Phi\in \Hscr_{1}^{N+1}$.
\end{prop}
\begin{proof}

On the one hand, by \eqref{orthogonal decomposition:L2},
we decompose $\Phi\in L^2\sts{\S^N}\setminus\ltl{0}$ as the following.
\begin{equation}\label{orth decomp Phi}
  \fct{\Phi}{\xi}
  =
  \sum_{k=0}^{\infty} 
\sum_{j=1}^{\dim{\Hscr_{k}^{N+1}}}
{\Phi_{k,j}}\fct{Y_{k,j}}{\xi},
\end{equation}
where
$\displaystyle 
  {\Phi_{k,j}}
  =
  \sint \fct{\Phi}{\xi}\fct{Y_{k,j}}{\xi}\diff{\xi}.
$

Combining  \eqref{IE sphere Phi}, \eqref{orth decomp Phi} with
 \Cref{double funk-hecke},  
we have for any $k\geq 0$ and $1\leq j\leq \dim{\Hscr_{k}^{N+1}}$, 
\begin{equation}
\label{Phi k equation}
  {\Phi_{k,j}}
  =
  \Const{G}\sts{N}{\bm\alpha}\sts{N,\lambda} 2^{-(p-1)(N-2)}
  \fct{\mu_k}{N-2}
  \left[{p\fct{\mu_k}{\lambda}+\sts{p-1}\fct{\mu_0}{\lambda}}\right]{\Phi_{k,j}}, 
\end{equation}

On the other hand, by \eqref{mu01 N-2} and \eqref{mu01 lambda},
by a direct computation, we obtain for $k=1$ that
\begin{align}
\label{mu 1 identity}
\Const{G}\sts{N}{\bm\alpha}\sts{N,\lambda} 2^{-(p-1)(N-2)} 
\fct{\mu_1}{N-2}
\left[{p\fct{\mu_1}{\lambda}+\sts{p-1}\fct{\mu_0}{\lambda}}\right]
=1,
\end{align}
which, together with \eqref{monotone mu k}, implies  for $k=0$ that
\begin{align}
\label{mu 0 >1}
\Const{G}\sts{N}{\bm\alpha}\sts{N,\lambda}  2^{-(p-1)(N-2)} 
\fct{\mu_0}{N-2}
\left[{p\fct{\mu_0}{\lambda}+\sts{p-1}\fct{\mu_0}{\lambda}}\right]
>1,
\end{align}
and for all $k\geq 2$ that
\begin{gather}\label{mu k <1}
\Const{G}\sts{N}{\bm\alpha}\sts{N,\lambda} 2^{-(p-1)(N-2)} 
\fct{\mu_k}{N-2}
\left[{p\fct{\mu_k}{\lambda}+\sts{p-1}\fct{\mu_0}{\lambda}}\right]
<1.
\end{gather}

Therefore, by inserting
\eqref{mu 1 identity}, \eqref{mu 0 >1} and \eqref{mu k <1}
into \eqref{Phi k equation}, we get for any $ k=0 $ and $k\geq 2$ that 
\begin{equation*}
  \Phi_{k,j}=0,
  \qtq{for } 1\leq j\leq \dim{\Hscr_{k}^{N+1}}.
\end{equation*}
Hence, we obtain that $\Phi\in \Hscr_{1}^{N+1}$ from \eqref{orth decomp Phi}.
\end{proof}

\begin{proof}[Proof of \Cref{theo2}]
Let $\varphi\in L^{\infty}(\R^N)$ satisfy \eqref{linear equation}.

Firstly, by \Cref{lem:lin eq to IE},  the function
$\varphi$ satisfies  \eqref{IE for linear equation}
and the estimate $\abs{\fct{\varphi}{x}}\lesssim \jp{x}^{-(N-2)}$.

Secondly, by \Cref{lem:IE sphere}, we have
${\Scal_{\ast}\varphi}\in L^2\sts{\S^N}$ satisfies
\eqref{IE sphere}. By Proposition \ref{lem:sltn to IE sphere},
we get ${\Scal_{\ast}\varphi}\in \Hscr_{1}^{N+1}$,
which together with \eqref{basis for eig 1} implies that
\begin{equation}
\label{S varphi in eig 1}
{\Scal_{\ast}\varphi}\in\mathrm{span}
  \ltl{\xi_j
  \,\middle\vert\,
  1\leq j\leq N+1
  }.
\end{equation}

Lastly, by Proposition \ref{prop:Rn to Sn}, we have
\begin{equation}
\varphi
\in
\mathrm{span}
\ltl{
\varphi_j
  \,\middle\vert\,
  1\leq j\leq N+1
  }.
\end{equation}
This completes the proof of \Cref{theo2}.
\end{proof}

\end{document}